\documentclass[10pt]{amsart}
\usepackage{geometry}                
\usepackage{graphicx}
\usepackage{amssymb}
\usepackage{hyperref}
\usepackage{float,subfigure}

\theoremstyle{plain} 
\newtheorem{theorem}{Theorem}[section]
\newtheorem{lemma}[theorem]{Lemma}
\newtheorem{corollary}[theorem]{Corollary}
\newtheorem{proposition}[theorem]{Proposition}
\theoremstyle{definition} 
\newtheorem{definition}[theorem]{Definition}

\newcommand{\R}{\mathbb{R}}

\newcommand{\re}{\operatorname{Re}}
\newcommand{\im}{\operatorname{Im}}

\begin{document}
\begin{titlepage}

\begin{title}
{Two New Embedded Triply Periodic Minimal Surfaces of Genus 4}
\end{title}

\author
{Daniel Freese}
\address{Daniel Freese\\Department of Mathematics\\Indiana University\\
Bloomington, IN 47405
\\USA}
\author{
Matthias Weber
}
\address{Matthias Weber\\Department of Mathematics\\Indiana University\\
Bloomington, IN 47405
\\USA}
\author
{A. Thomas Yerger}
\address{Alfred Thomas Yerger\\Department of Mathematics\\Indiana University\\
Bloomington, IN 47405
\\USA}
\author
{Ramazan Yol}
\address{Ramazan Yol\\Department of Mathematics\\Indiana University\\
Bloomington, IN 47405
\\USA}

\date{\today}
\maketitle

\begin{abstract}
We add two new 1-parameter families to the short list of  known embedded triply periodic minimal surfaces of genus 4 in $\R^3$.
Both surfaces can be tiled by minimal pentagons with two straight segments and three planar symmetry curves as boundary.
In one case (which has the appearance of the CLP surface of Schwarz with an added handle) the two straight segments are  parallel, while they are orthogonal in the second case. The second family has as one limit the Costa surface, showing that this limit can occur for triply periodic minimal surfaces. For the existence proof we solve the 1-dimensional period problem through a combination of an asymptotic analysis of the period integrals and geometric methods.
\end{abstract}

\end{titlepage}
\section{Introduction}

We construct two new, closely related 1-parameter families of embedded triply periodic minimal surface of genus 4 in Euclidean space. These surfaces are interesting for several reasons:

First, by a result of Meeks \cite{me6}, a triply periodic minimal surface of genus 4 cannot be hyperelliptic, limiting the known construction methods for these surfaces. In fact, the available list of examples is rather small: They consist of Alan Schoen's H'-T, I-WP, and S'-S'' surfaces \cite{sch1, ka5},
as well as several numerically constructed examples that  to the authors'  knowledge have never  been described in detail.

One particularly effective construction method that is still available is due to Traizet \cite{tr3a, tr7}: He is able to construct  triply periodic minimal surfaces of any genus $g>2$ that resemble horizontal planes joined by catenoidal necks.

Our surfaces, however, have more complicated limits. Indeed, one of the families
limits on one side in the Costa surface so that one could call it a {\em triply periodic Costa surface}. There exist other examples (of higher genus) with the appearance  of a triply periodic Costa surface (see Batista's surface \cite{Bat1}  and Alan Schoen's I6 surface, called {\em Figure 8 annulus} in \cite{ka5}), but these examples do not truly limit in the Costa surface but rather in the singly periodic Callahan-Hoffman-Meeks surface (\cite{chm2}). This is significant if one wants to extend Traizet's regeneration construction to
employ more general necks than the catenoidal ones: Our example suggests it should be possible to use Costa necks joining three consecutive planes. A Callahan-Hoffman-Meeks limit would require an entirely different gluing procedure, involving cutting off a Callahan-Hoffman-Meeks surface by a cylinder, glued to the complement of a solid vertical cylinder in a family of horizontal planes at finite distance from each other.

Finally, our surfaces are examples of the only two possible types of genus 4 triply periodic minimal surfaces that have the vertical coordinate planes as symmetry planes and the line $y=x$ in the plane $z=0$ part on the surface.

\def\fw{2.5in}
\begin{figure}[H]
 \begin{center}
   \subfigure[CS-4]{\includegraphics[width=\fw]{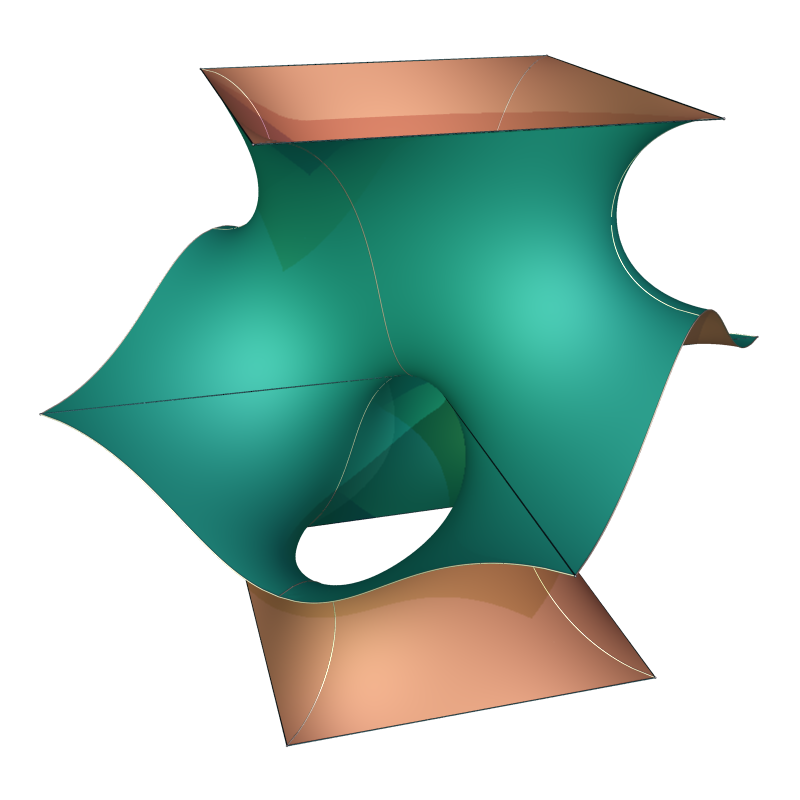}}
   \subfigure[SS-4]{\includegraphics[width=\fw]{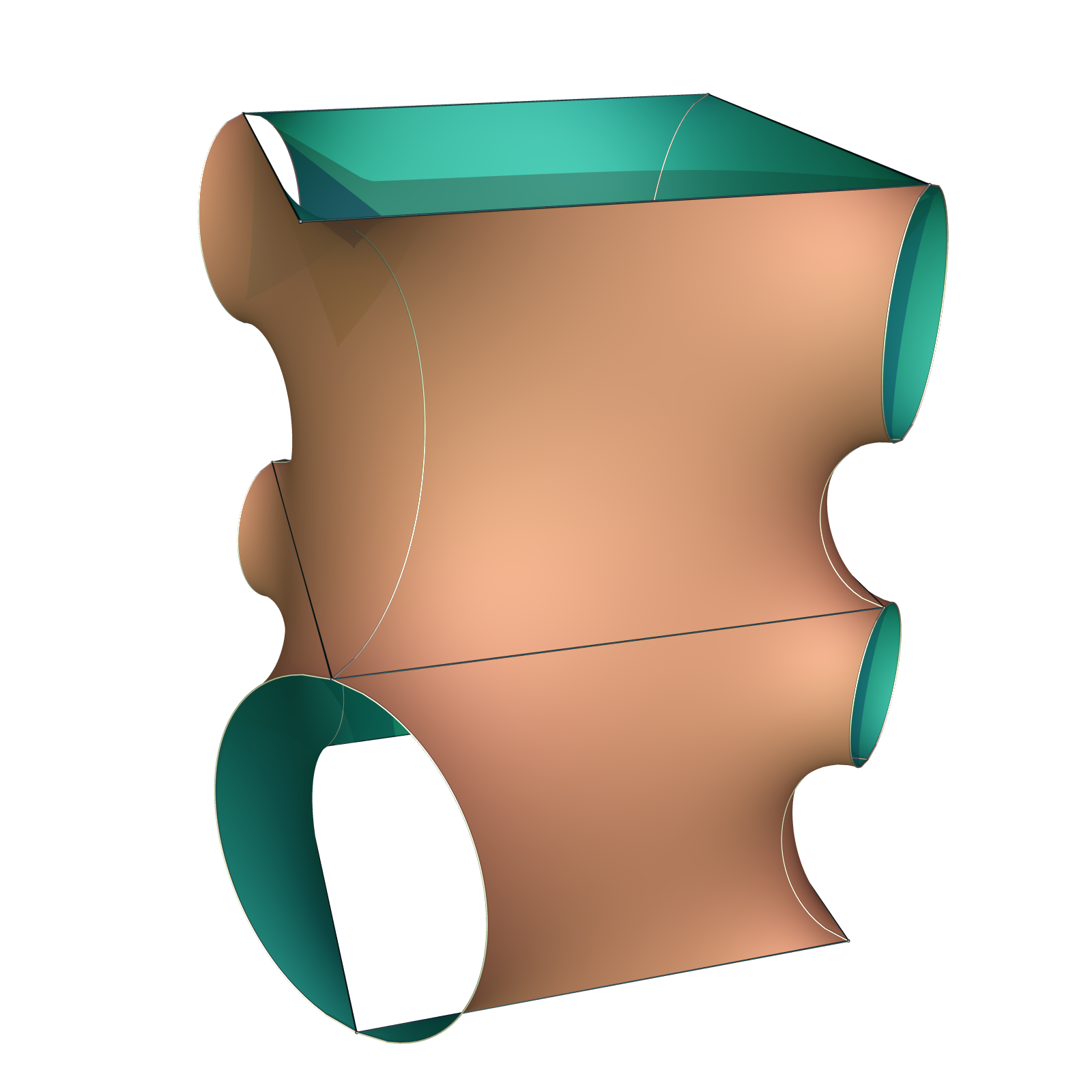}}
 \end{center}
 \caption{Translational Fundamental Pieces}
 \label{fig:fundamental}
\end{figure}

To state  our main results, we introduce some notation.
Let $\Pi$ be a minimal pentagon in a box $[0,1]\times [0,1]\times [0,h]$ where two edges are diagonals of the top and bottom faces of the box, respectively, and the remaining three edges lie in the vertical faces of the box. Moreover, all vertices of the pentagon lie on box edges or are box vertices. We further assume that the normal vector along the pentagon edges that lies in the vertical faces of the box also lies in the plane of these faces, making them symmetry planes.

\begin{definition}
We say that any minimal pentagon $\Pi$ satisfying these conditions  is of type SS if the two horizontal segments are parallel, and of type CS if they are orthogonal.
\end{definition}

\begin{theorem} \label{thm:penta}
There exist two 1-parameter families of minimal pentagons of type SS and CS, respectively.
\end{theorem}

In order to extend the minimal pentagon $\Pi$, we first rotate it about its diagonal in the top face, then extend  by reflection at the front and right side of the box to obtain 8 copies of the pentagon that constitute a translational fundamental piece of a triply periodic minimal surface $\tilde X$. We denote the quotient of $\tilde X$ by the translational symmetries by $X$. This is  a genus 4 Riemann surface. 

\begin{corollary} 
The two families of pentagons from Theorem \ref{thm:penta} extend to embedded triply periodic minimal surfaces of genus 4. These surfaces have orthogonal vertical symmetry planes over a square grid and horizontal straight diagonals.
\end{corollary}

Examples of the surfaces obtained this way can be seen in Figure \ref{fig:fundamental}.  


\def\fw{3.in}
\begin{figure}[H]
 \begin{center}
   \subfigure[near Costa surface]{\includegraphics[width=\fw]{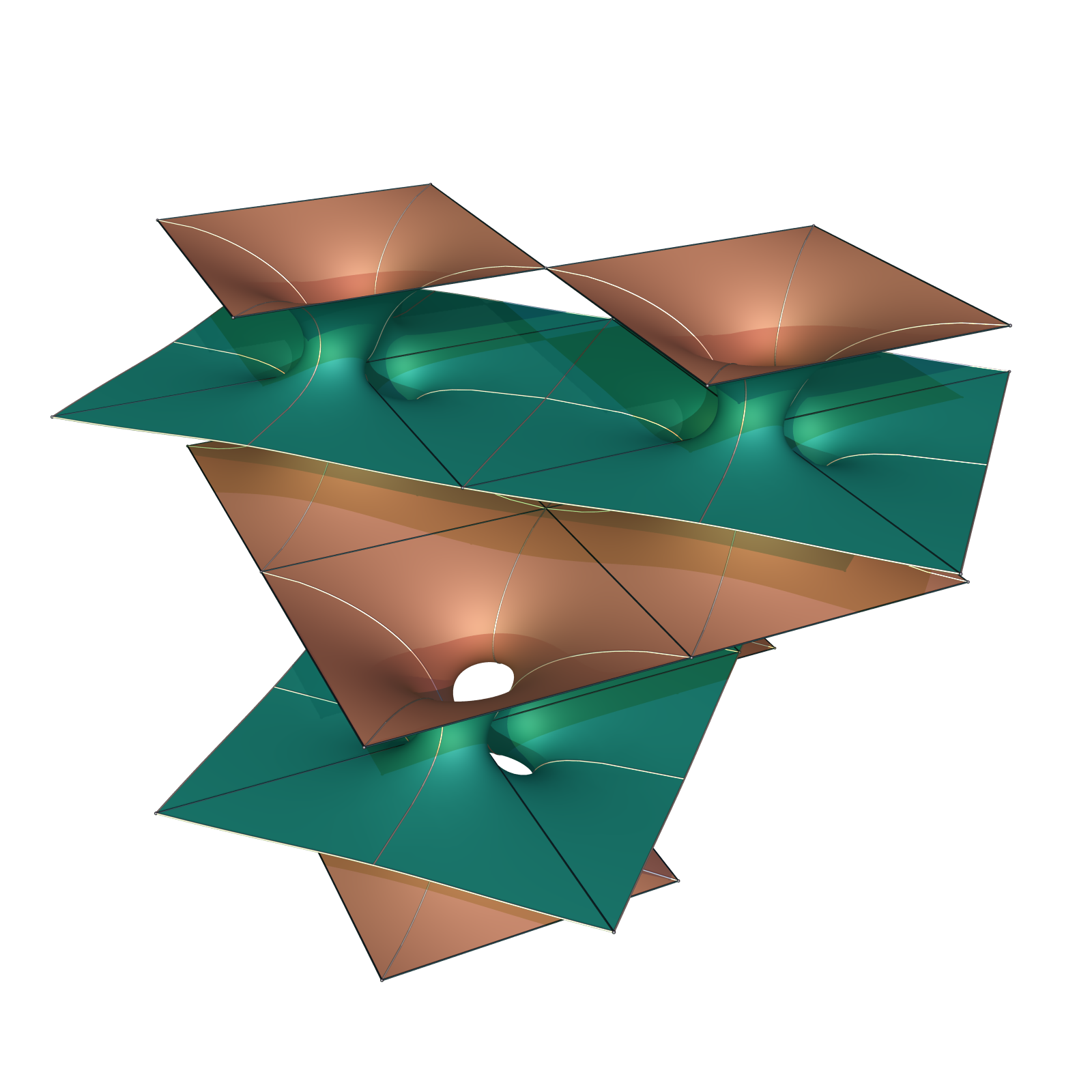}}
   \subfigure[near doubly Scherk surfaces]{\includegraphics[width=\fw]{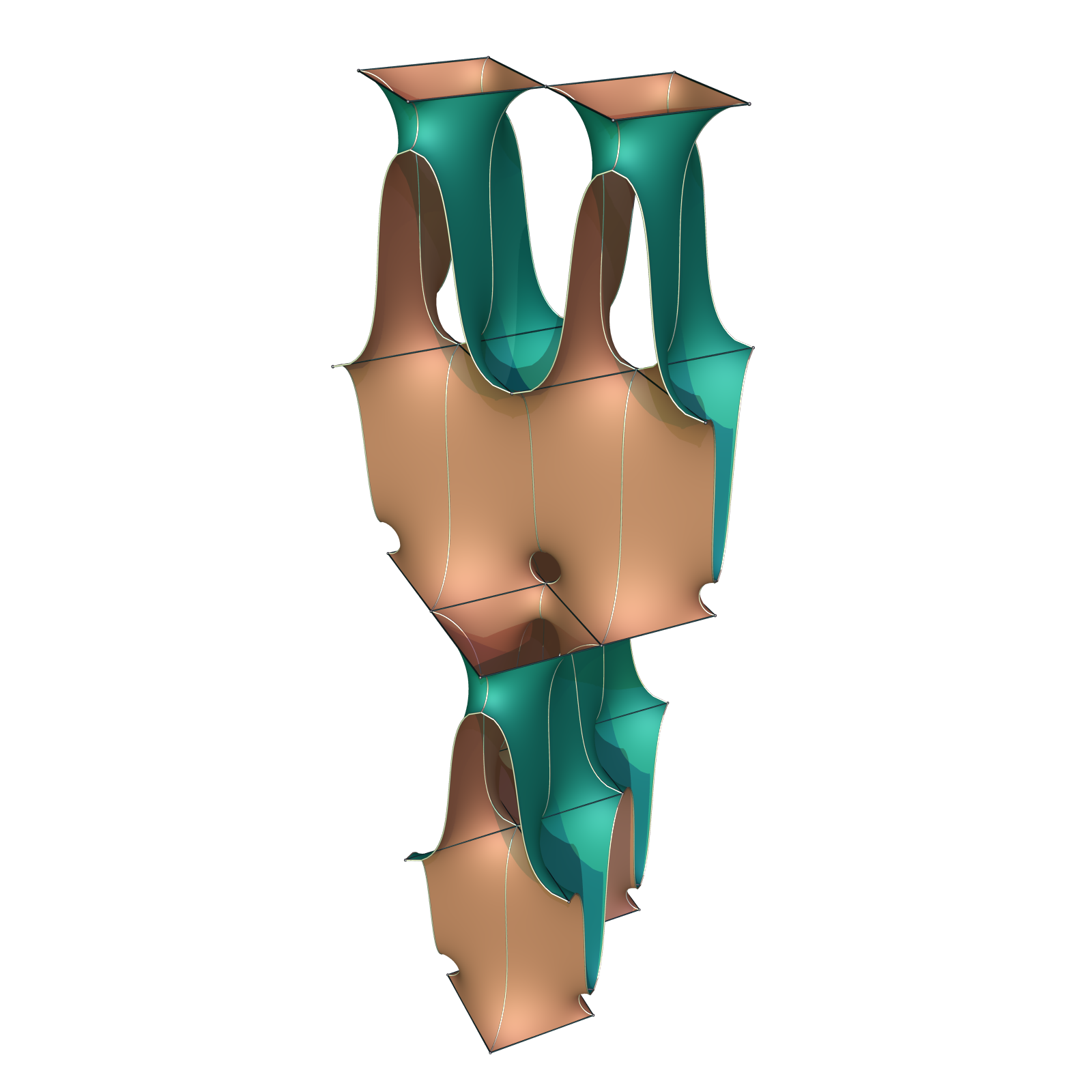}}
 \end{center}
 \caption{Limits of the CS-4 surfaces}
 \label{fig:limcs4}
\end{figure}

Regarding the limits of our surfaces, we will prove:
\begin{theorem} \label{thm:limitcs}
There is a sequence of  CS-4 surfaces converging to the Costa surface.
\end{theorem}

Numerical evidence suggests that  other limits of CS-4 surfaces include  the doubly periodic Scherk and  Karcher-Scherk surfaces.

Moreover,
one limit of the SS-4 surface appears to consist of a family of vertical planes over a square grid desingularized by two different singly periodic Scherk surfaces, one having twice the translational period of the other. 
Such desingularizations have been constructed by \cite{tr1,tr3a}.
At the other end of the parameter range, two interesting limits appear to be  possible, namely the doubly periodic Scherk surface (with orthogonal ends), or the Karcher-Scherk surface of genus 1 \cite{howe3}.

\def\fw{3.in}
\begin{figure}[H]
 \begin{center}
   \subfigure[near singly Scherk surfaces]{\includegraphics[width=\fw]{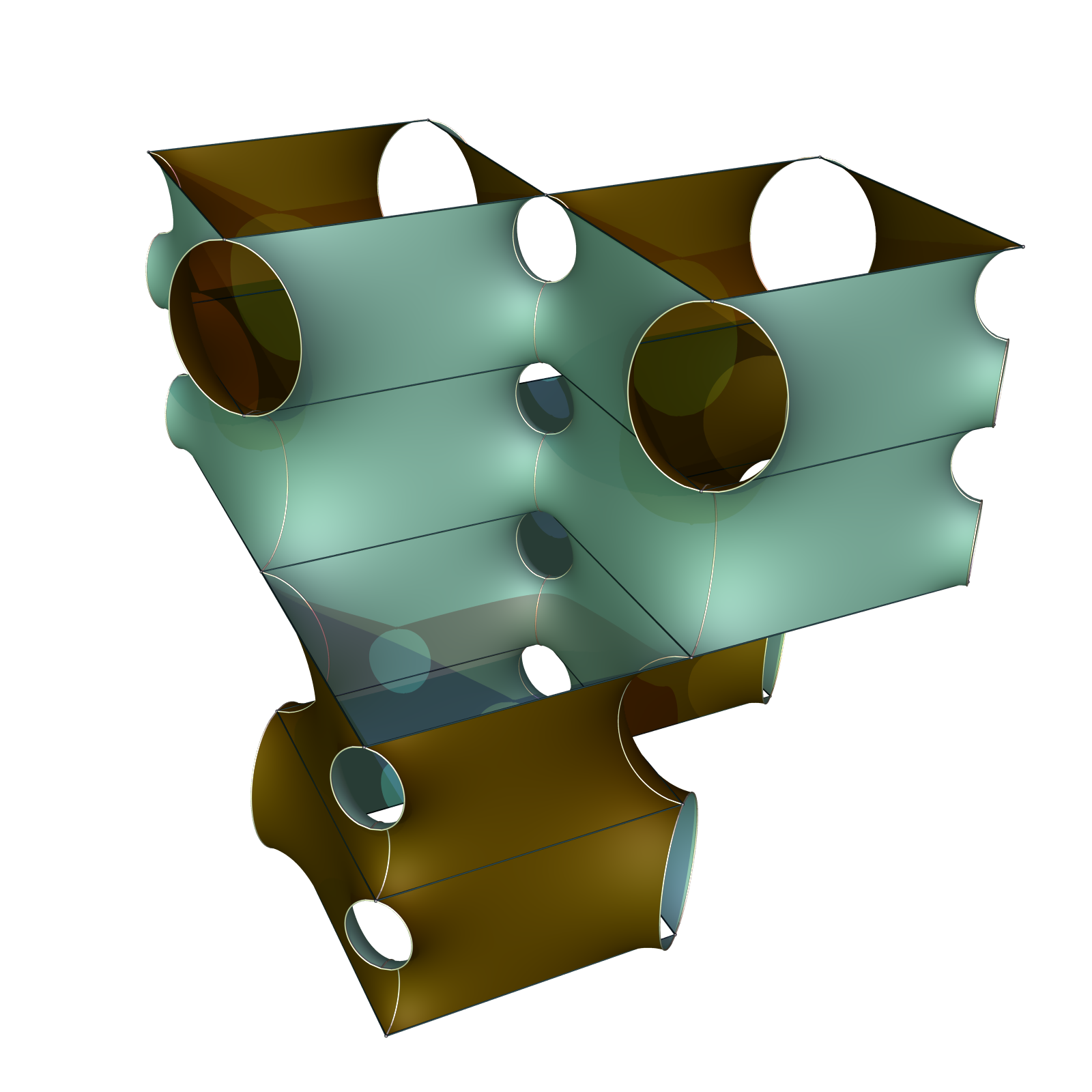}}
   \subfigure[near doubly Scherk surfaces]{\includegraphics[width=\fw]{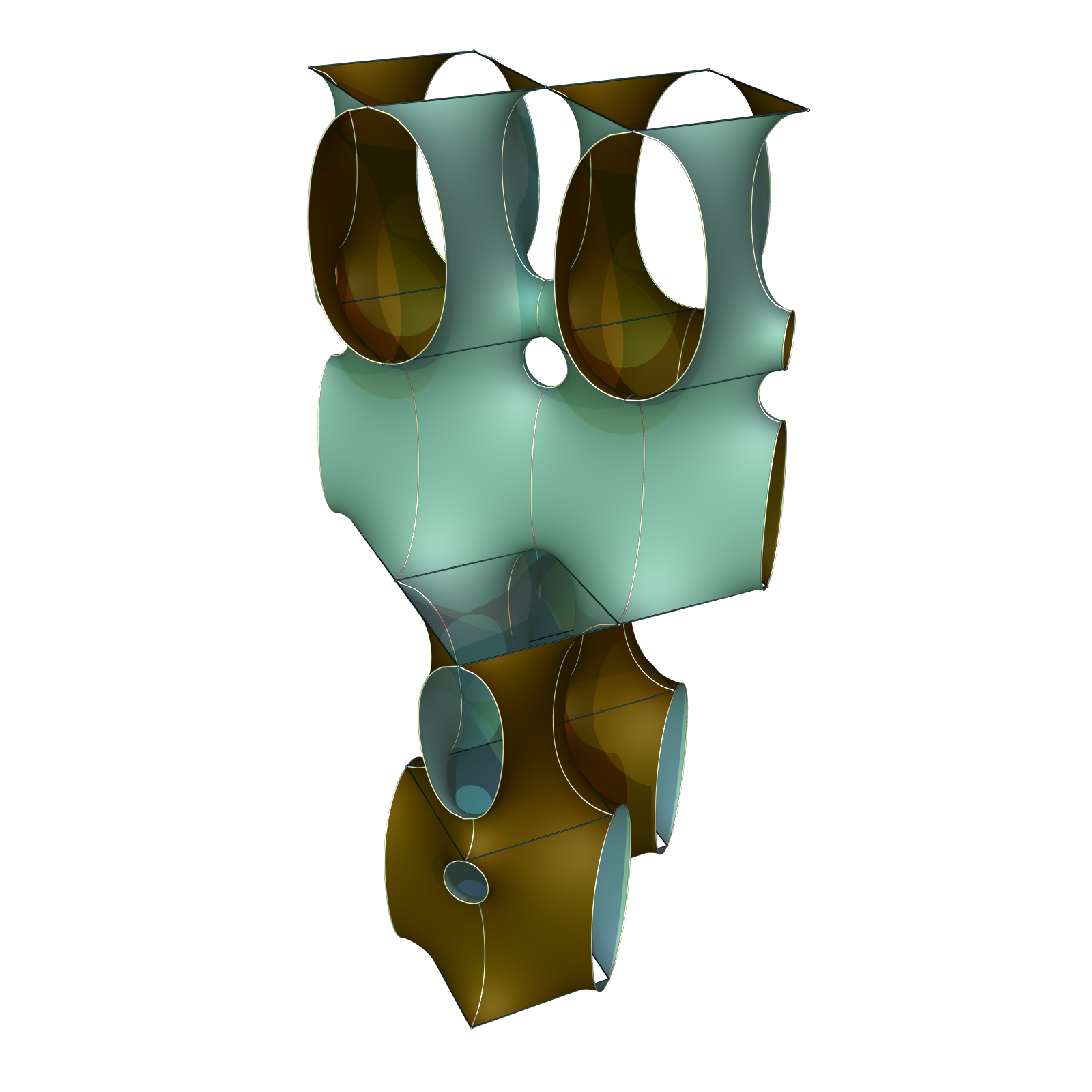}}
 \end{center}
 \caption{Limits of the SS-4 surfaces}
 \label{fig:limss4}
\end{figure}

The paper is organized as follows: In Section \ref{sec:geom}, we show that the assumed symmetries of our surfaces have strong implications on the flat structures of the Weierstrass 1-forms, allowing us to parametrize the surfaces via Schwarz-Christoffel maps.
Section \ref{sec:ss4} is devoted to the existence proof of the SS-4 family, and Section \ref{sec:cs4}  deals the CS-4 family.
In Section \ref{sec:limits} we prove that the SS-4 surfaces limit in the Costa surface.
Embeddeness is proven in Section \ref{sec:embed}. Finally,  in Section \ref{sec:sym}, we characterize the CS-4 and SS-4 surface families by their symmetries.

\section{Geometry of the Weierstrass Representation}\label{sec:geom}

Let a minimal map (i.e. a conformal parametrization of a minimal surface) be given by

\[
f(z) = \re \int^z (\omega_1, \omega_2, \omega_3)
\]
where

$$ \omega_1 = {}  \frac12 (\frac1G-G)\, dh, \qquad \omega_2 = {}  \frac{i}2 (\frac1G+G)\, dh, \qquad \omega_3= {} dh \ .$$

Here, the meromorphic function $G$ is the stereographic projection of the Gauss map, and the holomorphic 1-form $dh$ is called the height differential.

Recall that multiplying $dh$ by a real factor scales the surface, and multiplying it by $e^{i t}$ is the Bonnet deformation. 
Multiplying $G$ by a real factor is called the L{\'o}pez-Ros deformation, while multiplying $G$ by $e^{i t}$ rotates the surface about a vertical axis by the angle $\varphi$.


Let $f:U \to \R^3$ be a minimal map, given by Weierstrass data $G$ and $dh$.
Introduce  $\Omega_k(z) =\int^z \omega_k$,  $\Phi_1(z) =\int^z G\, dh$, and $\Phi_2(z) =\int^z \frac1G\, dh$.

We will next explain that the particular symmetries we assume about our surfaces imply that the flat structures of $dh$ , $G\, dh$ and $\frac1G\, dh$ are Euclidean pentagons. This is crucial for our line of reasoning, because it will allow us to define  $dh$ , $G\, dh$ and $\frac1G\, dh$ as integrands of Schwarz-Christoffel maps from the upper half plane to such Euclidean pentagons.

\begin{proposition}
\label{prop:symmetries}
Suppose that $\Omega_3$, $\Phi_1$, and $\Phi_2$ extend continuously to  the real interval $(a,b)\subset \partial U$ and map it to a segment orthogonal to, to a segment making angle $\alpha$ with, and to a segment making angle $-\alpha$ with the real axis, respectively. Then the Schwarz reflection principle guarantees that $\Omega_3$, $\Phi_1$, and $\Phi_2$ and thus $f$ can be extended across $(a,b)$ by reflection. We claim that this extension of $f$ is realized by a $180^\circ$ rotation  about  $f(a,b)$, which is a horizontal straight line in $\R^3$ making angle $\alpha$ with the $x$-direction.
\end{proposition}
\begin{proof}
To see this, we first note that we can assume that $\alpha=0$. Otherwise we multiply $G$ by $e^{-\alpha i}$: This rotates the segments $\Phi_1(a,b)$ and $\Phi_2(a,b)$ to become parallel to the real axis, and leaves $\Omega_3(a,b)$ unchanged. On the other hand, it rotates the surface about a vertical  axis by angle $-\alpha$.

Since now $\alpha=0$, extension across $(a,b)$  conjugates $G\, dh$  and $\frac1G\, dh$. Consequently,
this leaves $\re\omega_1$ unchanged, while it turns   $\re\omega_2$  and $\re\omega_3$ into $-\re\omega_2$ and $-\re\omega_3$.

Vice versa, if a minimal surface contains a horizontal straight line (necessarily a symmetry line) that is parametrized by a segment $(a,b)\subset \R$, the Weierstrass integrals above map $(a,b)$ to segments with the appropriate angles.

Similarly,    $\Omega_3$, $\Phi_1$, and $\Phi_2$ map the real interval $(a,b)\subset U$ to a  segment parallel to, a segment making angle $\alpha$ with, and a segment making angle $-\alpha$ with the real axis, respectively, if and only if $f(a,b)$ is a reflectional symmetry curve in a vertical plane making angle $\alpha$ with the $x$-direction.
\end{proof}

\section{The SS-4 Surface}\label{sec:ss4}

\begin{figure}[H]
   \centering
   \includegraphics[width=6in]{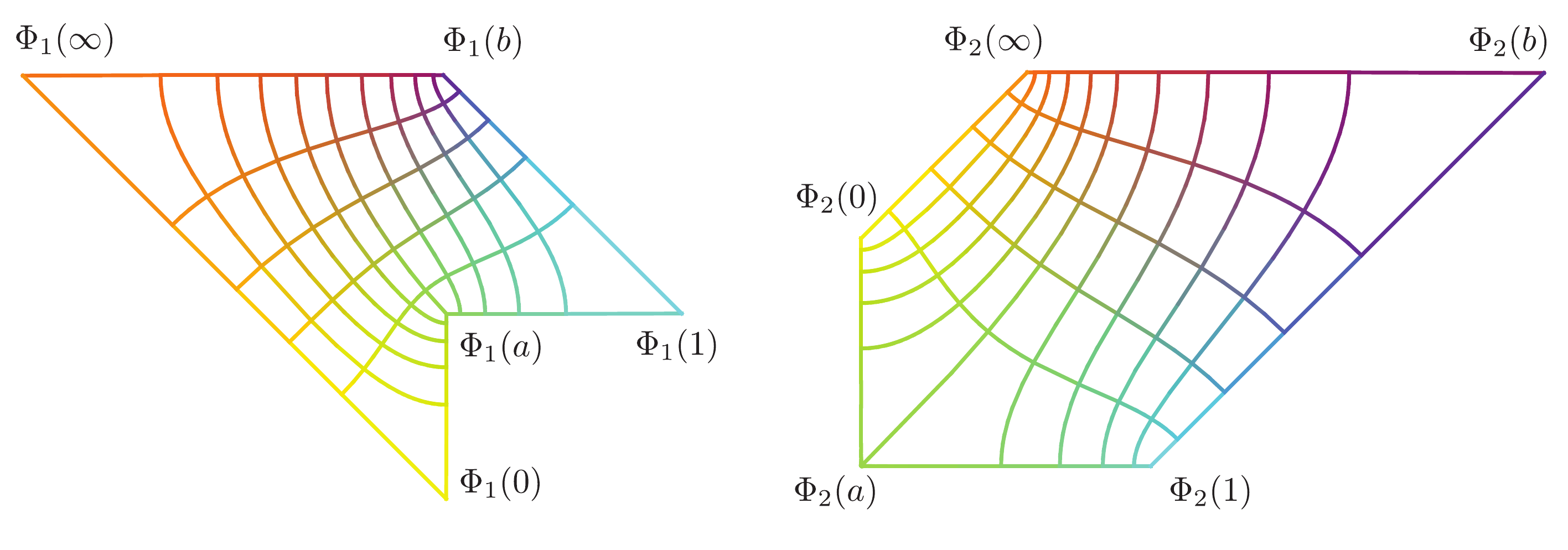}
   \caption{Schwarz-Christoffel images of Upper Half Plane for SS-4}
   \label{fig:SS-4-schwarz}
\end{figure}

We begin the  construction of the SS-4 surfaces using Schwarz-Christoffel maps to the polygons in Figure \ref{fig:SS-4-schwarz}. Consider in the upper half plane the 1-forms

\begin{align*}
&\varphi_1 = z^{-3/4}(z-a)^{1/2}(1-z)^{-3/4}(b-z)^{-1/4}dz\\
&\varphi_2 =  z^{-1/4}(z-a)^{-1/2}(1-z)^{-1/4}(b-z)^{-3/4}dz
\end{align*}

where $0< a < 1 <b$. Then $\varphi_1$ and $\varphi_2$ are positive on $(a,1)$, and extend analytically. In particular, the interval $(a,1)$ is mapped to a segment parallel to the real axis by $\Phi_1$ and $\Phi_2$, as in the figure.

These forms will become the Weierstrass 1-forms $G\, dh= \rho \varphi_1$ and $\frac1G\, dh =\frac1\rho \varphi_2$ after choosing a  suitable L{\'o}pez-Ros  factor $\rho$. Consequently, the height differential is given by
\[
dh = z^{-1/2}(1-z)^{-1/2}(b-z)^{-1/2}dz \ .
\]
Note that  this integrand is again a Schwarz-Christoffel integrand, and positive on $(a,1)$. This implies that $\Omega_3$ maps the upper half plane to a rectangle with the images of the segments $(-\infty, 0)$ and $(1,b)$ becoming segments orthogonal to the real axis. By Proposition \ref{prop:symmetries}, $f$ will  map these segments the horizontal straight lines, and all other segments to curves in symmetry planes making the required angles.

We have proven:

\begin{lemma}\label{lem:symss}
For any $\rho>0$ and $0<a<1<b$, the Weierstrass map $f$
maps the upper half plane to a minimal pentagon with three edges lying in vertical symmetry planes and two horizontal edges  which are straight line segments.  One symmetry plane has its normal parallel to the $x$-axis, and two have normals parallel to the $y$-axis. The horizontal lines make angle $-\frac{\pi}{4}$ with $x$-axis.
\end{lemma}

We will next discuss the period problem for these surfaces.

\begin{proposition}
A pentagon as constructed in Lemma \ref{lem:symss} is of type $SS$ if and only if
\[
\frac{\int_a^1|\varphi_1|}{\int_0^a|\varphi_1|}=\frac{\int_a^1|\varphi_2|}{\int_0^a|\varphi_2|} \ .
\]
\end{proposition}
\begin{proof}
In order for the pentagon to lie in a box over a square as claimed in Theorem \ref{thm:penta} we need the vertices $f(0)$, $f(a)$, $f(1)$ to be on the same vertical line as in Figure   \ref{fig:FundamentalPentagons} (a). For, this, it suffices to have $\re \int_0^1 (\omega_1, \omega_2)=0$,  because it forces $f(b)$ and $f(\infty)$ to be on the same vertical line. We also  observe that automatically $\re \int_0^a \omega_1=0$ since $f(0)$ and $f(a)$ are on the same symmetry plane with $x$ coordinate  fixed.  Similarly $\re \int_a^1 \omega_2=0$. So $\re \int_0^1 (\omega_1, \omega_2 )=0$ if and only if $\re \int_0^a \omega_2=0$ and $\re \int_a^1 \omega_1=0$ i.e. $\im \rho \int_0^a \varphi_1=-\im \frac1\rho \int_0^a \varphi_2$ and $\re \rho \int_a^1 \varphi_1=\re\frac1\rho \int_a^1 \varphi_2$. 
Now it is clear that by choosing a positive L{\'o}pez-Ros factor we can force either one of last two  equalities and we know that  $\int_0^a \varphi_1$ is positive imaginary and $\int_a^1 \varphi_1$ is positive real.  Thus  our condition is equivalent to the claimed equality.
\end{proof}

\def\fw{2.8in}
\begin{figure}[H]
 \begin{center}
   \subfigure[SS-4 Pentagon]{\includegraphics[width=\fw]{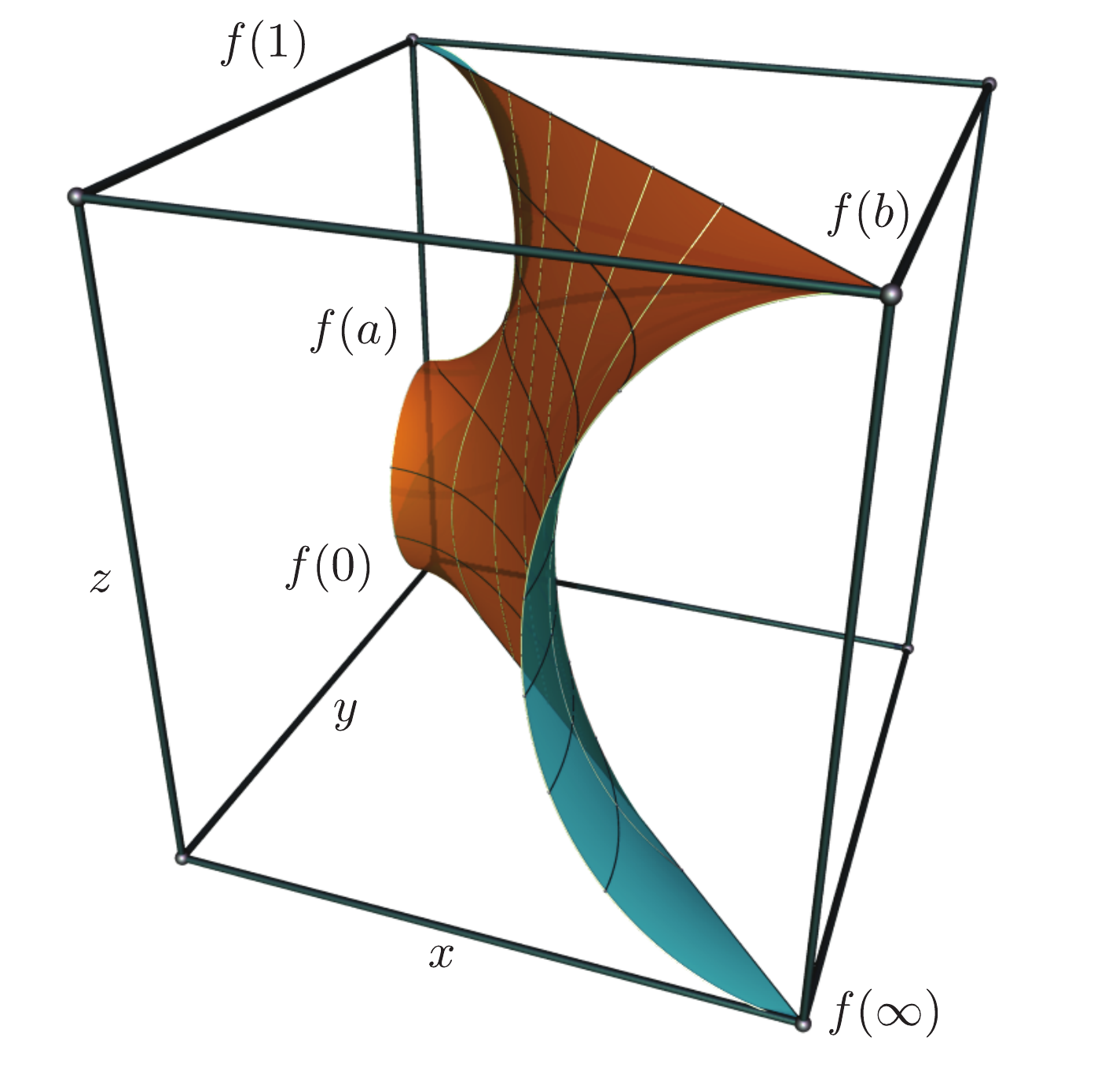}}
   \subfigure[CS-4 Pentagon]{\includegraphics[width=\fw]{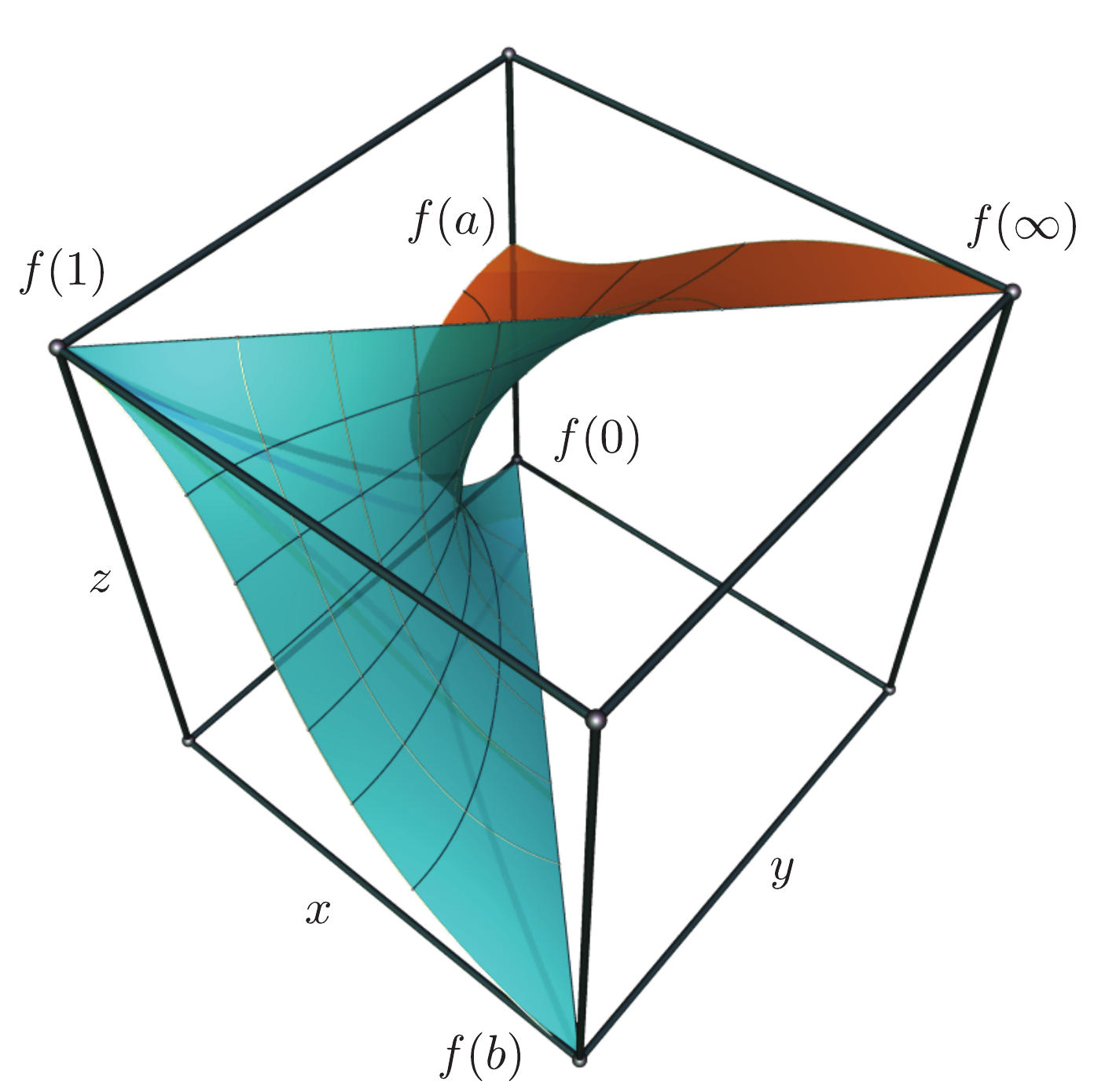}}
 \end{center}
 \caption{Fundamental Pentagons}
 \label{fig:FundamentalPentagons}
\end{figure}

The following Theorem shows that there is a 1-parameter family of solutions to this period condition, proving the first part of Theorem \ref{thm:penta}.
\begin{theorem}
 For any $b>1$ there is an $a\in(0,1)$ satisfying

\[
\pi_1:= \frac{\int_a^1|\varphi_1|}{\int_0^a|\varphi_1|}=\frac{\int_a^1|\varphi_2|}{\int_0^a|\varphi_2|} =:\pi_2
\]

\end{theorem}

\begin{proof} We fix $b$ and use an intermediate value argument, comparing the behavior of $\pi_1$ and $\pi_2$ at $0$ and $1$, see Figure \ref{fig:ss4-asymptotic}. Here, as well as in the corresponding proof for the CS-4 surface, we consider the asymptotic behavior of the period integrals, using the following convention:

\begin{definition}
Given two positive real valued functions  $f$ and  $g$ defined near some $a\in \R\cup \{\infty\}$, we say $f \approx g$ as $x\to a$ if 

\[
\frac1Cg(x) \leq f(x) \le Cg(x)
\] 

for some $C>0$ and $x$ approaching $a$.\\

\end{definition}

\begin{figure}[H]
   \centering
   \includegraphics[width=3in]{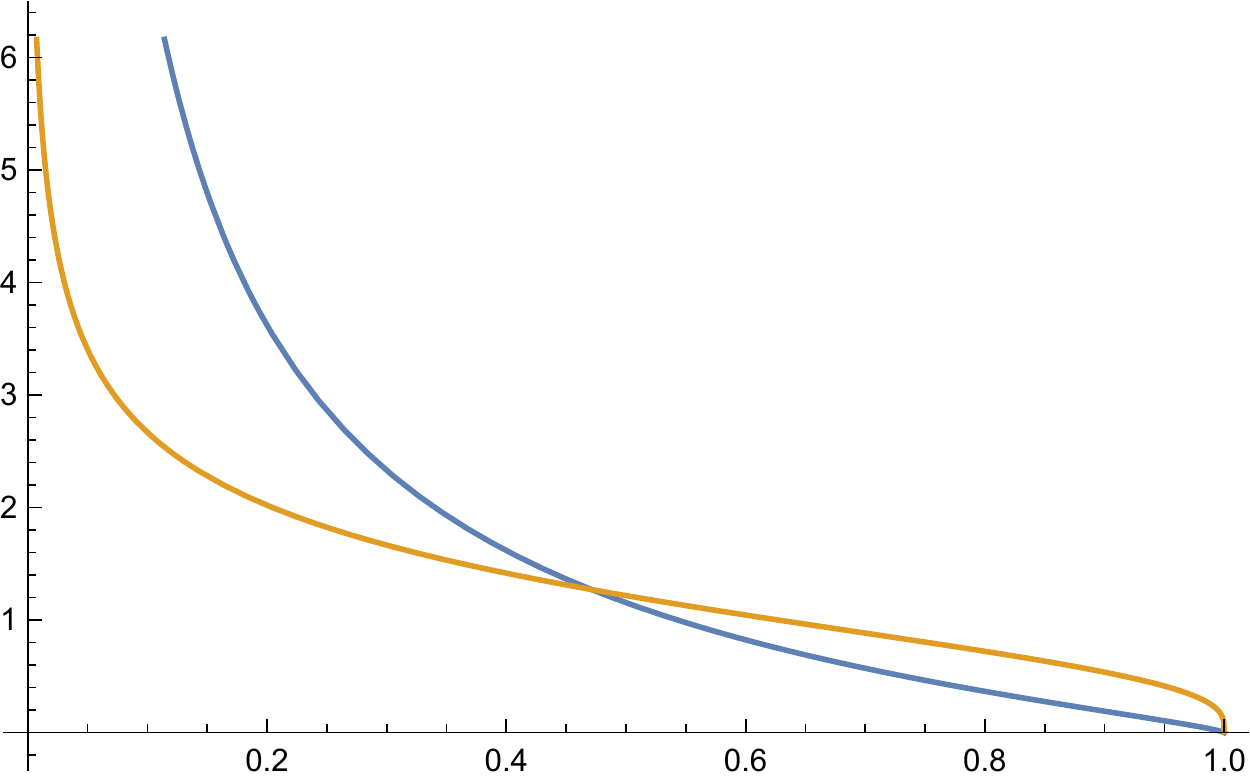}
   \caption{Graphs of $\pi_1$ and $\pi_2$ as functions of $a$ for $b$ fixed}
   \label{fig:ss4-asymptotic}
\end{figure}

First, note that our integrals only involve $0 < z <1$, in which case $|z - b| = b - z$ and $0<b-1<|z-b|<b$.  Thus $|z-b|$ is bounded away from 0 and hence $\approx 1$.\\

As $a\to0$, we have on $(a,1)$:

\begin{flalign*}
\int_a^1|\varphi_1|&\approx \int_a^1 z^{-3/4}(z-a)^{1/2}(1-z)^{-3/4}\, dz&&\\
&=\int_a^{1/2} z^{-3/4}(z-a)^{1/2}(1-z)^{-3/4}dz + \int_{1/2}^1 z^{-3/4}(z-a)^{1/2}(1-z)^{-3/4}\, dz&&\\
&\approx \int_a^{1/2} z^{-3/4}(z-a)^{1/2}dz + \int_{1/2}^1 (1-z)^{-3/4}\, dz&&\\
&\approx 1 + 1
\end{flalign*}

\begin{flalign*}
\int_a^1|\varphi_2|&\approx \int_a^1 z^{-1/4}(z-a)^{-1/2}(1-z)^{-1/4}\, dz&&\\
&\approx 1
\end{flalign*}

Next, on $(0,a)$:

\begin{flalign*}
\int_0^a|\varphi_1|&\approx \int_0^a z^{-3/4}(a-z)^{1/2}(1-z)^{-3/4}\, dz&&\\
&\approx \int_0^a z^{-3/4}(a-z)^{1/2}dz&&\\
&= a^{-3/4 + 1/2 + 1}\int_0^1t^{-3/4}(1-t)^{1/2}\, dt \qquad \text{using $z=at$}&&\\
&\approx a^{3/4}
\end{flalign*}

\begin{flalign*}
\int_0^a|\varphi_2|&\approx \int_0^a z^{-1/4}(a-z)^{-1/2}(1-z)^{-1/4}\, dz&&\\
&\approx a^{1/4}
\end{flalign*}

Thus, as $a\to0$, we have,

\[
\frac{\int_a^1|\varphi_1|}{\int_0^a|\varphi_1|}\approx a^{-3/4}\gg a^{-1/4}\approx\frac{\int_a^1|\varphi_2|}{\int_0^a|\varphi_2|}
\]

Using similar arguments we obtain for  $a\to1$

\begin{align*}
 \int_a^1|\varphi_1|  \approx {}& (1-a)^{3/4} & \int_0^a|\varphi_1|  \approx {}& 1 \\
\int_a^1|\varphi_2|  \approx {}& (1-a)^{1/4} & \int_0^a|\varphi_2|  \approx {}& 1 \\       
\end{align*}

%
%
%
%

Hence, as $a\to1$

\[
\frac{\int_a^1|\varphi_1|}{\int_0^a|\varphi_1|}\approx (1-a)^{3/4}\ll (1-a)^{1/4}\approx\frac{\int_a^1|\varphi_2|}{\int_0^a|\varphi_2|}
\]

%

By the intermediate value theorem, for any fixed $b>1$ there is an $a\in(0, 1)$ making the ratios equal.
\end{proof}

\section{The CS-4 Surface}\label{sec:cs4}

\begin{figure}[H]
   \centering
   \includegraphics[width=6in]{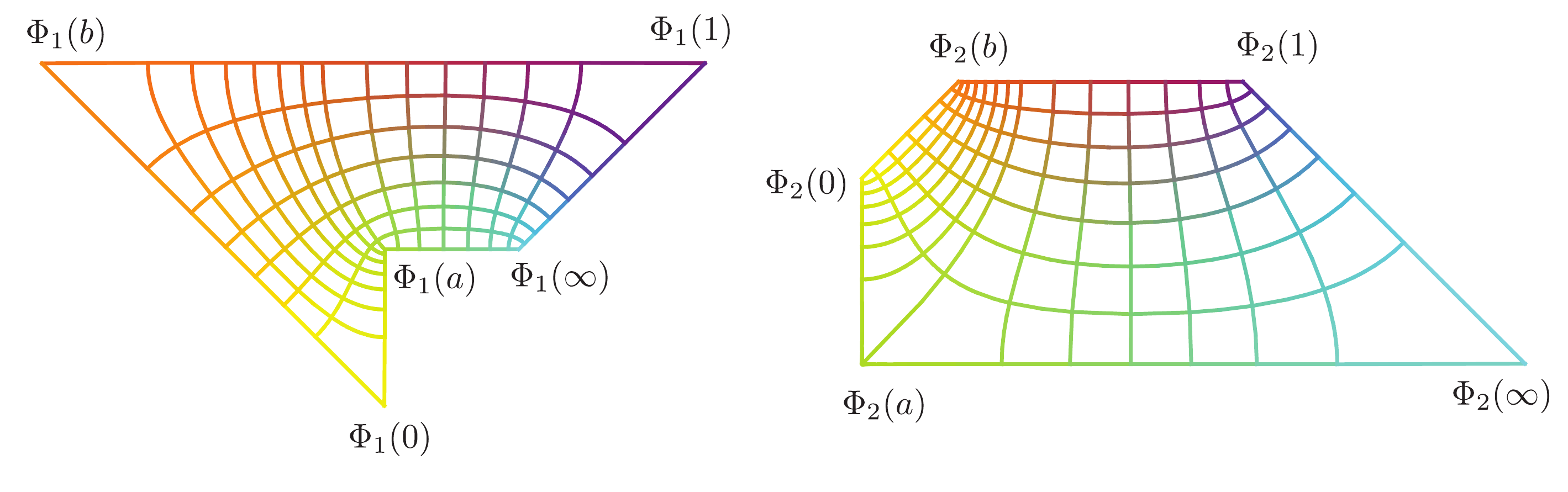}
   \caption{Schwarz-Christoffel images of Upper Half Plane for CS-4}
\label{fig:CS-4-schwarz}
\end{figure}

As with the SS-4 surface, we construct the CS-$4$ surface using Schwarz-Christoffel maps. We define:

\begin{align*}
&G\, dh = \rho(z-a)^{1/2}z^{-3/4}(z-b)^{-3/4}(1-z)^{-3/4}\, dz\\
&\frac1G\, dh = \frac1\rho(z-a)^{-1/2}z^{-1/4}(z-b)^{-1/4}(1-z)^{-1/4}\, dz
\end{align*}

where $a < 0 < b < 1$. Note that the ordering of the parameters is different from the SS-4 case. Here, $\varphi_1$ and $\varphi_2$ are positive on $(b,1)$.

These forms become the Weierstrass 1-forms $G\, dh $ and $\frac1G\, dh$ after scaling by a real factor. Consequently, the height differential is given by
\[
dh = z^{-1/2}(z-b)^{-1/2}(1-z)^{-1/2}\, dz
\]
and the Gauss map by
\[
G(z) = \rho(z-a)^{1/2}z^{-1/4}(z-b)^{-1/4}(1-z)^{-1/4}
\]

Then we have as in Section \ref{sec:ss4}:

\begin{lemma}
For any $\rho>0$ and $a < 0 < b < 1$, the Weierstrass map $f(z)$
maps the upper half plane to a minimal pentagon with three edges lying in vertical symmetry planes and two horizontal edges  which are straight line segments. One symmetry plane has its normal parallel to the $x$-axis, and two have normals parallel to the $y$-axis.  One horizontal line makes angle $\frac{\pi}4$ with the $x$-axis, and the other makes angle  $-\frac{\pi}4$.
\end{lemma}

For such a pentagon to extend via symmetries to a triply-periodic surface, the vertices of our pentagon must lie on the edges of a box as shown in Figure \ref{fig:FundamentalPentagons}(b).  This translates to the following period condition:  The vertices $f(a)$ and $f(0)$ must have the same $y$-coordinate, and the vertices $f(0)$ and $f(1)$ must have the same $x$-coordinate.

 Our forms must thus satisfy:

 \[
\re\int_a^0\omega_2 = 0, \quad\text{and}\quad \re\int_0^1\omega_1 = 0
 \]

\begin{figure}[H]
   \centering
   \includegraphics[width=6in]{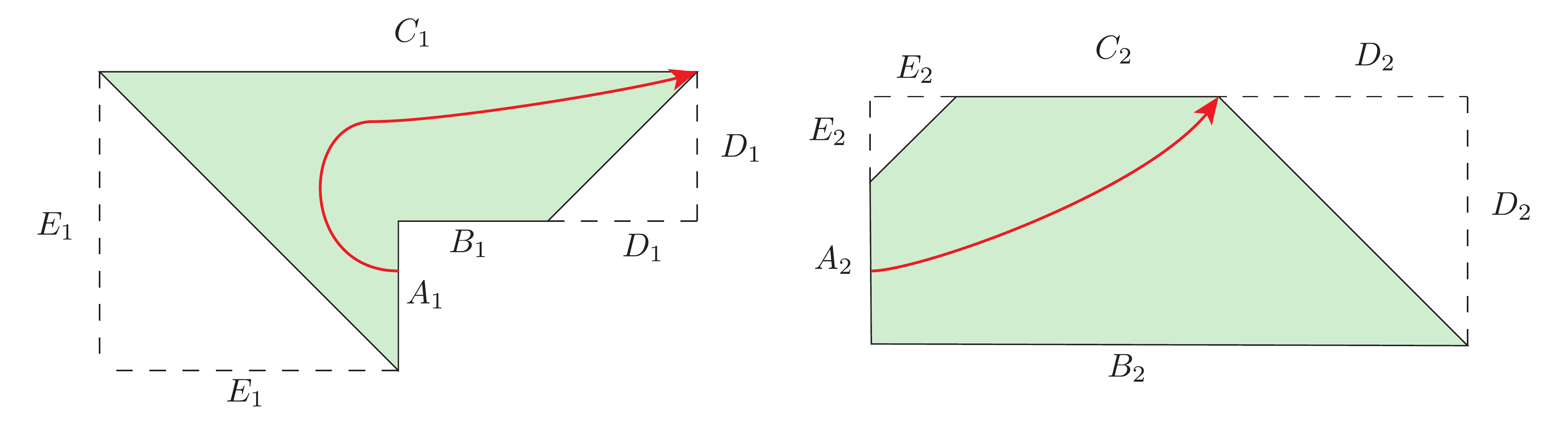}
   \caption{Edge lengths and cycles for CS-4}
   \label{fig:CS-4-lengths}
\end{figure}

For simplicity, we will represent these integrals in terms of side lengths of the Schwarz-Christoffel pentagons shown in Figure \ref{fig:CS-4-lengths}.

\begin{lemma}
There exists a unique positive L{\'o}pez-Ros factor $\rho$ such that 
the surface parametrized by the Weierstrass data above is of type CS   if and only if the flat structures satisfy  
\[
\sigma_1:=\frac{B_1 + C_1}{A_1} = \frac{B_2 + C_2}{A_2}=:\sigma_2 \ .
\]
\end{lemma}

\begin{proof}
In terms of $\varphi_1, \varphi_2$, the  condition $\re\int_a^0\omega_2 = 0$  translates to $\im\int_a^0\varphi_1 = -\im\int_a^0\varphi_2$.  Since these integrals are both imaginary with opposite signs, this is equivalent to $A_1 = A_2$.  \\

The  condition $\re\int_0^1\omega_1 = 0$ translates to $\re\int_0^1\varphi_1 = \re\int_0^1\varphi_2$.  On the pentagon flat structures, this is equivalent to $B_1 + D_1 = C_2 + E_2$, as can be seen in Figure \ref{fig:CS-4-lengths}.

The first condition can be satisfied by choosing $\rho$, in which case the second condition becomes $\frac{B_1 + D_1}{A_1} = \frac{C_2 + E_2}{A_2}$.  For simplicity, we will express this condition just in terms of $A, B$, and $C$ lengths.  For this, we use the relations of the flat structure edge lengths:

\[
D_1 = E_1 - A_1 = C_1 - D_1 - B_1 -A_1 \Rightarrow D_1 = \frac12(C_1 - B_1 - A_1)
\]

\[
E_2 = D_2 - A_2 = B_2 - C_2 - E_2 - A_2 \Rightarrow E_2 = \frac12(B_2 -C_2 - A_2)
\]

The condition can thus be written as:

\[
\frac{2B_1 + C_1 - B_1 - A_1}{2A_1} = \frac{2C_2 + B_2 - C_2 - A_2}{2A_2} \Leftrightarrow \frac{B_1 + C_1}{A_1} = \frac{B_2 + C_2}{A_2}
\]

which completes the proof.

\end{proof}

We next prove the second half of Theorem \ref{thm:penta}:

\begin{theorem}\label{thm:cs4exist}
For any $a < 0$ there is a $b_a\in (0, 1)$ guaranteeing the equality $\frac{B_1 + C_1}{A_1} = \frac{B_2 + C_2}{A_2}$. Thus there is a one-parameter family of surfaces which satisfy the period condition.
\end{theorem}

\begin{figure}[H]
   \centering
   \includegraphics[width=3in]{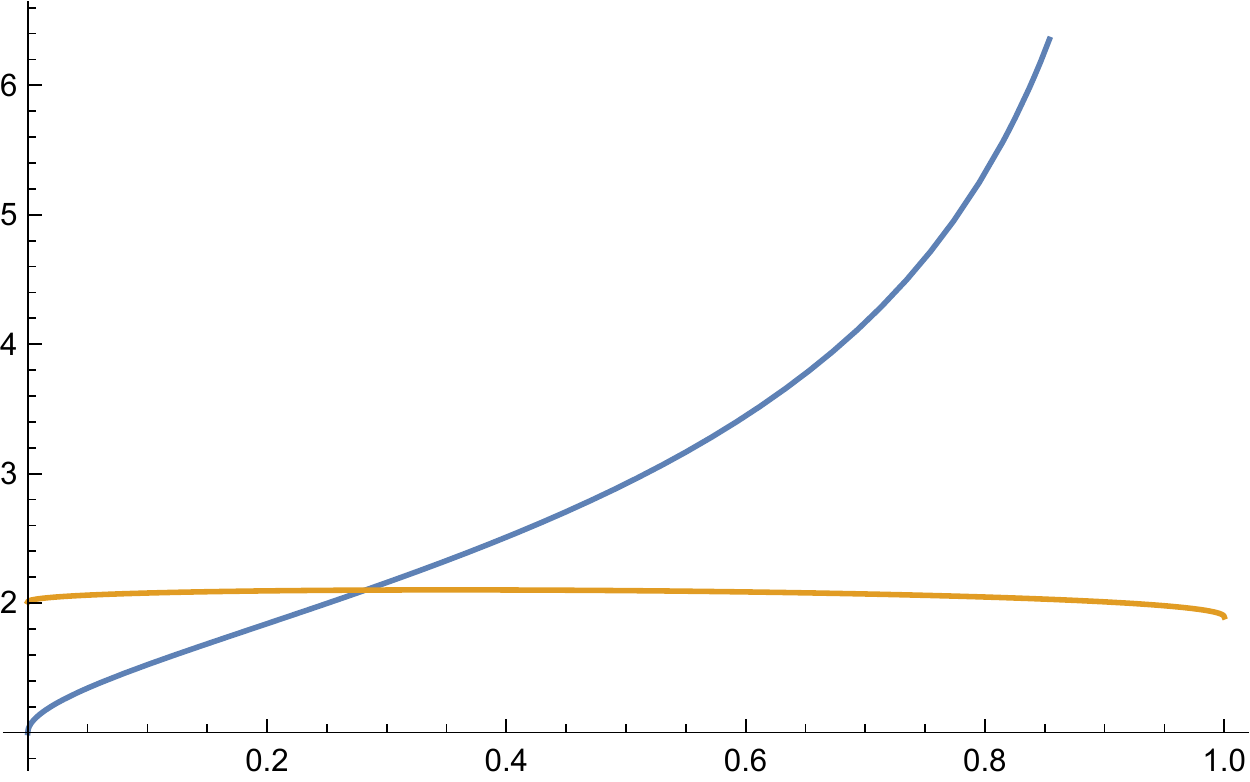}
   \caption{Graphs of $\sigma_1$ and $\sigma_2$ as functions of $b$ for $a=-3$}
   \label{fig:cs4-asymptotic}
\end{figure}

\begin{proof}

We fix $a<0$ and first send $b \to 1-$. We consider the effect on the integrals $A_1$,$ B_1$, $C_1$, $A_2$, $B_2$, and $C_2$.

First, $A_1\approx 1$ and $A_2\approx 1$, because $b$  and $1$ are outside the integration interval $[a,0]$.

%
%
%

For $B_1$ and $B_2$, $(b-z)^{-3/4}(1-z)^{-3/4}$ is bounded on the integration interval $(-\infty,a]$, so that again
$B_1\approx 1$ and $B_2\approx 1$.

%
%
%

For the $C_1, C_2$ integrals, we obtain
\begin{flalign*}
C_1 &= \int_b^1 (z - a)^{1/2} z^{-3/4}(z-b)^{-3/4}(1-z)^{-3/4} dz&& \\
&\approx \int_b^1 (z-b)^{-3/4}(1-z)^{-3/4} dz\\
&= (1-b)^{-1/2}\int_0^1(1 - t)^{-3/4}t^{-3/4}dt \qquad \text{using}\ z=1 - (1-b)t\\
&\approx (1-b)^{-1/2} \ ,
\end{flalign*}
and similarly
\[
C_2 \approx (1-b)^{1/2} \ .
\]

Thus as $b\to 1$, we have (see Figure \ref{fig:cs4-asymptotic})

\[
\frac{B_1 + C_1}{A_1}\approx (1-b)^{-1/2}\gg 1 \approx \frac{B_2 + C_2}{A_2} \ .
\]

Now we consider the case $b\to 0$.

For $A_1$ and $A_2$, singularities occur on both ends of the integration interval $[a,0]$ which we therefore split up and treat separately.

Note that for $a < z < \frac a2$, $-z \approx b - z \approx 1 - z \approx 1$ and for $\frac a2 < z < 0$, $z - a \approx 1 - z \approx 1$. Thus

\begin{flalign*}
A_1 &= \int_a^0 (z-a)^{1/2}(-z)^{-3/4}(b-z)^{-3/4} (1-z)^{-3/4} dz&&\\
&\approx \int_a^{a/2}  (z - a)^{1/2} dz + \int_{a/2}^0 (-z)^{-3/4}(b-z)^{-3/4} dz\\
\end{flalign*}

Setting $z = -bt$, we obtain $b-z = b (1+t)$, allows us to estimate further:

\begin{flalign*}
A_1 &\approx 1 +  b^{-1/2} \int_0^{-a/2b} t^{-3/4}(1+t)^{-3/4} dt &&\\
&=1+b^{-1/2} \bigg( \int_0^1 t^{-3/4}(1+t)^{-3/4} dt + \int_1^{-a/2b} t^{-3/4} (1+t)^{-3/4} dt \bigg)
\end{flalign*}

The first term in brackets can be integrated and thus is $\approx 1$, and for the second term we use the observation that $1 < t  < 1+t < 2t$

\begin{flalign*}
\int_1^{-a/2b} t^{-3/4} (1+t)^{-3/4} &\approx \int_1^{-a/2b} t^{-3/2} dt&&\\
&\approx \sqrt b
\end{flalign*}

This shows $A_1 \approx 1 + b^{-1/2}(1 + b^{1/2}) \approx b^{-1/2}$ as $b \to 0$.

For $A_2$, the analysis is similar:

\begin{flalign*}
A_2 &= \int_a^0 (z - a)^{-1/2} (-z)^{-1/4}(b-z)^{-1/4} (1-z)^{-1/4} dz &&\\
&\approx 1 + \int_{a/2}^0 (-z)^{-1/4}(b-z)^{-1/4} dz \\
&\approx 1 + b^{1/2}(1 + b^{-1/2}) \approx 1 \\
\end{flalign*}


For $B_1$and $B_2$, we split at $2a-1$ like so:

\begin{flalign*}
B_1 &= \int_{-\infty}^a (a - z)^{1/2} (-z)^{-3/4}(b-z)^{-3/4}(1-z)^{-3/4} \, dz&& \\
&\approx \int_{-\infty}^{2a-1}(-z)^{-7/4} \, dz+  \int_{2a-1}^a (a - z)^{1/2} (-z)^{-3/4}(b-z)^{-3/4}(1-z)^{-3/4} \, dz\\
&\approx 1 \ ,
\end{flalign*}

and likewise

\begin{flalign*}
B_2 &= \int_{-\infty}^a (a - z)^{-1/2} (-z)^{-1/4}(b-z)^{-1/4}(1-z)^{-1/4} \, dz &&\\
&= \int_{-\infty}^{2a-1} (-z)^{-5/4}\, dz + \int_{2a-1}^a (a - z)^{-1/2} (-z)^{-1/4}(b-z)^{-1/4}(1-z)^{-1/4} \, dz \\
& \approx 1
\end{flalign*}

For  $C_1, C_2$, we split at $\frac12$:

\begin{flalign*}
C_1 &= \int_b^1 (z - a)^{1/2} z^{-3/4}(z-b)^{-3/4}(1-z)^{-3/4} \,dz&& \\
&\approx \int_{b}^{1/2} z^{-3/4}(z-b)^{-3/4}\, dz + \int_{1/2}^1 (1-z)^{-3/4} dz \\
&\approx  b^{-1/2} \int_{1}^{\frac1{2b} }t^{-3/4}(t-1)^{-3/4}dt + 1  \qquad \text{using }\ z=bt\\
\end{flalign*}


Because of

\[
\int_{1}^{2}t^{-3/4}(t-1)^{-3/4}dt \leq \int_{1}^{1/2b}t^{-3/4}(t-1)^{-3/4}dt \leq \int_{1}^{\infty}t^{-3/4}(t-1)^{-3/4}dt
\]

 $C_1\approx b^{-1/2}$. A similar computation shows  $C_2\approx 1$.

%

We thus have that for $b\to 0$, both period quotients are bounded and bounded away from zero:

\[
\frac{B_1 + C_1}{A_1} \approx \frac{1 + b^{-1/2}}{b^{-1/2}} \approx 1, \qquad \frac{B_2 + C_2}{A_2} \approx \frac{1 + 1}1 = 1
\]

To obtain a more accurate comparison between these ratios, we consider the flat structures corresponding to $\varphi_1$ and $\varphi_2$, and argue geometrically together with the estimates above.
  
We claim:

\begin{lemma}\label{lem:bclose0}
For $b$ sufficiently close to 0, we  have
\[
\frac{B_1 + C_1}{A_1}<\frac{B_2 + C_2}{A_2} \ .
\]
\end{lemma}

For the left hand side note that   $\frac{B_1}{A_1} \to 0$. We proceed to show that $\frac{C_1}{A_1} \to 1$. 

 From  Figure \ref{fig:CS-4-lengths}, we see that

\begin{align*}
E_1 &= A_1 + D_1\\
C_1 &= E_1 + B_1 + D_1
\end{align*}

Thus,

$$\frac{C_1}{A_1} = \frac{A_1 + B_1 + 2D_1}{A_1} = 1 + \frac{B_1}{A_1} + 2\frac{D_1}{A_1}$$

and it remains to show $\frac{D_1}{A_1} \to 0$.

Splitting the integration interval $[1,\infty)$ at 2 allows us to estimate
\begin{flalign*}
D_1 &= \frac1{\sqrt{2}}\int_1^\infty (z - a)^{1/2} (z)^{-3/4}(z - b)^{-3/4}(z - 1)^{-3/4} dz&& \\
&\approx \int_1^2 (z - 1)^{-3/4}dz + \int_2^\infty z^{-7/4}\\
&\approx 1 \ .
\end{flalign*}

Thus $\frac{D_1}{A_1} \approx \frac1{b^{-1/2}} \to 0$, which proves $\frac{B_1+C_1}{A_1} \to 1$.

For the $\varphi_2$ integrals, we note that $A_2, B_2$, and $C_2$ are all bounded away from 0, so $\frac{C_2}{A_2} \ge L>0$ for some number $L$ and $b$ sufficently close to 0.  From Figure \ref{fig:CS-4-lengths}, we note that $B_2$ must always be larger than $A_2$.  Hence $\frac{B_2 + C_2}{A_2} \ge 1 + L >1$.\par

%
%

This proves our claim and  completes the proof of the theorem.

\end{proof}

\section{Limits}\label{sec:limits}

In this section we will prove Theorem \ref{thm:limitcs}. More precisely, we will show that there is a sequence $(a_i,b_i) \to (-\infty, 0)$ solving he period problem such that the corresponding CS surfaces, suitably scaled, converge the Costa surface.


In order to so, we will need the following lemma:

\begin{lemma}\label{lem:cslim}
For any $b\in (0, 1)$ and $N>0$, there is an $a<-N$ such that $\frac{B_1 + C_1}{A_1} > \frac{B_2 + C_2}{A_2}$.
\end{lemma}

Using this lemma, we prove the existence of the Costa limit as follows:

\begin{proof}(of Theorem \ref{thm:limitcs}) Fix  $N>0$ large and $\epsilon >0$ as usual.
We begin by picking $a_0=-N$. By Theorem \ref{thm:cs4exist}, there exists $b_0$ such that the period problem is solved for $(a_0, b_0)$. Now,  Lemma \ref{lem:bclose0} implies that we can find  $0<b_1<\epsilon$ so  that for the pair $(a_0, b_1)$ 
\[
\frac{B_1 + C_1}{A_1}<\frac{B_2 + C_2}{A_2} \ .
\]

Applying Lemma \ref{lem:cslim} above to $b_1$, we can find  $a_0'<a_0=-N$ so that the reverse inequality holds for $(a_0', b_1)$.
By the intermediate value theorem there is $a_1\in (a_0',a_0)$ such that $(a_1, b_1)$ solves the period problem. Thus there are solutions $(a,b)$  to the period problem with $a$ arbitrarily negative and $b$ arbitrarily small.



Recall the forms 
\begin{align*} 
  G\, dh  = {}&  \rho\varphi_1 =  \rho(a-z)^{-1/2}(z)^{-3/4}(b-z)^{-3/4}(1-z)^{-3/4} \, dz \\
  \frac1G\, dh = {}&\frac1\rho\varphi_1 = \frac1\rho  (a-z)^{-1/2}(z)^{-1/4}(b-z)^{-1/4}(1-z)^{1/4} \, dz
\end{align*}
where $\rho$ is determined by
\[
\rho^2 = \frac {\int_0^a\varphi_2}{\int_0^a \varphi_1}
\]
such that $A_1=A_2$.

Taking the limit as $a \to - \infty$ and $b \to 0$, we obtain the limit forms $\psi_1 = \lim \varphi_1= \sqrt{-a}z^{-3/2}(z-1)^{-3/4}\, dz$ and $\psi_2= \lim \sqrt{-a} \varphi_2 =z^{-1/2} (z-1)^{-1/4}\, dz$. Defining $\rho$ by taking limits as well 
\[
\rho^2 = \frac {\int_0^{-\infty}\psi_2}{\int_0^{-\infty} \psi_1}
\]
we obtain Weierstrass data for a limit surface.

We claim that these are the Weierstrass data for the Costa surface. This follows from Example 3.5.1 in \cite {ww2}, where it is proven that the flat structures of the 1-forms $G\, dh$ and $\frac1G\, dh$ of the Costa surface are the infinite polygons shown in Figure \ref{fig:Costa-lengths}. Moreover,  the periods for the Costa surface are closed if and only if $C_2=D_1$.

\begin{figure}[H]
   \centering
   \includegraphics[width=6in]{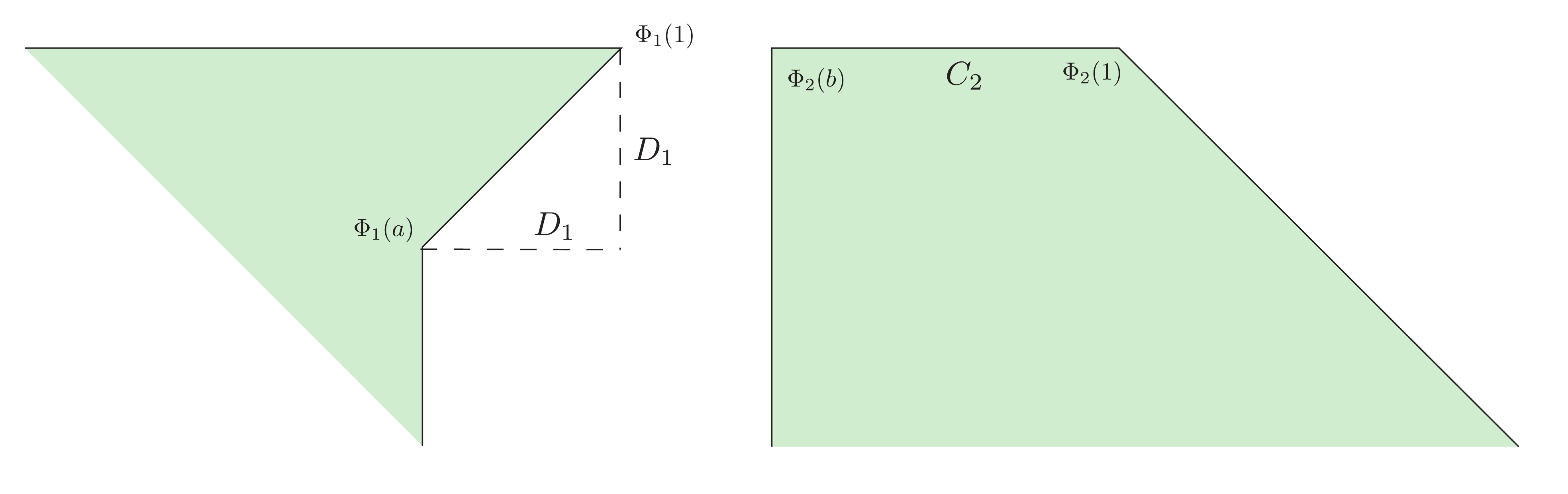}
   \caption{Edge lengths for the flat structures of the Costa surface}
   \label{fig:Costa-lengths}
\end{figure}

To see that these are the limit flat structures of our limit 1-forms, note that first $A_2\to \infty$ and $B_2, E_2\to 0$ by convergence of the Schwarz-Christoffel maps. By the period conditions, this implies that also $A_1\to \infty$ and that $D_1$ and $C_2$ have the same finite limit as required.
This implies that the constructed limit is the Costa surface.

\end{proof}

\begin{proof}(of Lemma \ref{lem:cslim})

Suppose that $b$ is fixed, and send $a \to - \infty$. As usual, we estimate:

\begin{flalign*}
A_1 &= \int_a^0 (z - a)^{1/2} (-z)^{-3/4}(b-z)^{-3/4}(1-z)^{-3/4} dz &&\\
&\approx \int_a^{-1} (z - a)^{1/2} (-z)^{-3/4}(b-z)^{-3/4}(1-z)^{-3/4} dz + \int_{-1}^0 (z - a)^{1/2} (-z)^{-3/4}\,  dz\\
\end{flalign*}
The first integral is bounded above by 
\[
 (-a)^{1/2} \int_a^{-1}  (-z)^{-3/4}(b-z)^{-3/4}(1-z)^{-3/4} dz  \approx (-a)^{1/2}
 \]
 while the second is $\approx (-a)^{1/2}$. Hence the sum is also $\approx (-a)^{1/2}$.

Using the substitution $z = at$, the usual cancellations, we obtain $A_2 \approx (-a)^{-1/4} $.
\medskip

For estimates of $B_1$ and $B_2$, note that when $z < a < -1$, we have $-2z > 1 - z > b - z > -z$ and thus $(b - z)\approx (1 - z) \approx -z$, giving,

\begin{flalign*}
B_1 &= \int_{-\infty}^a (a - z)^{1/2} (-z)^{-3/4}(b-z)^{-3/4}(1-z)^{-3/4} dz&& \\
&\approx \int_{-\infty}^a (a - z)^{1/2} (-z)^{-9/4}dz \\
&\approx (-a)^{-3/4}\int_0^1\left(\frac1t - 1\right)^{1/2}\left(\frac1t\right)^{-1/4}dt \text{   using } z = \frac at\\
&\approx (-a)^{-3/4} \\
\end{flalign*}

\begin{flalign*}
B_2 
&\approx \int_{-\infty}^a (a - z)^{-1/2} (-z)^{-3/4}dz &&\\
&\approx (-a)^{-1/4}\int_0^1\left(\frac{1-t}t \right)^{-1/2}\left(\frac1t\right)^{5/4}dt\\
&\approx (-a)^{-1/4}\int_0^1(1-t)^{-1/2}t^{-3/4}dt\\
&\approx (-a)^{-1/4}
\end{flalign*}

For the estimates of $C_1$ and $C_2$: 

\begin{flalign*}
C_1 &= \int_b^1 (z - a)^{1/2} z^{-3/4}(z - b)^{-3/4}(1-z)^{-3/4} dz&& \\
&\approx -(a)^{1/2}\int_b^1 z^{-3/4}(z - b)^{-3/4}(1-z)^{-3/4} dz \\
&\approx (-a)^{1/2}
\end{flalign*}

\begin{flalign*}
C_2 &= \int_b^1 (z-a)^{-1/2} z^{-1/4} (z - b)^{-1/4}(1-z)^{-1/4} dz &&\\
&\approx (-a)^{-1/2}\int_b^1  z^{-1/4} (z - b)^{-1/4}(1-z)^{-1/4} dz \\
&\approx (-a)^{-1/2}
\end{flalign*}

Our ratios of interest become

\begin{align*}
\frac{B_1 + C_1}{A_1} \approx \frac{(-a)^{-3/4} + (-a)^{1/2}}{(-a)^{1/2}}\approx 1\\
\frac{B_2 + C_2}{A_2} \approx \frac{(-a)^{-1/4} + (-a)^{-1/2}}{(-a)^{-1/4}}\approx 1
\end{align*}

Thus the ratios $\frac{B_1 + C_1}{A_1}$ and $\frac{B_2 + C_2}{A_2}$ both remain bounded away from 0 as $a\to -\infty$.  To compare them, we consider the geometry of the flat structures.

For the $\varphi_1$ structure, we have $\frac{B_1}{A_1}\approx (-a)^{-5/4} \to 0$ and $\frac{C_1}{A_1} \approx 1$.  We claim that $\frac{C_1}{A_1} \to 1 + L$ where $L > 0$.\\

%

Using the lengths of the polygons,

$$\frac{C_1}{A_1} = \frac{A_1 + B_1 + 2D_1}{A_1} = 1 + \frac{B_1}{A_1} + 2\frac{D_1}{A_1}$$

and it remains to show $\frac{D_1}{A_1} \ge L > 0$ for $a\to -\infty$.

\begin{flalign*}
D_1 &\approx\int_1^\infty (z - a)^{1/2} (z)^{-9/4}\, dz &&\\
&= (-a)^{1/2}\int_1^\infty (1-z/a)^{1/2} (z)^{-9/4}\, dz \\
&\approx (-a)^{1/2}
\end{flalign*}

Thus $\frac{D_1}{A_1} \approx 1$, which proves the claim.

For the $\varphi_2$ structure, we have $\frac{B_1}{A_2}\approx1$ and $\frac{C_2}{A_2} \approx (-a)^{-1/4} \to 0$.  We claim that $\frac{B_2}{A_2} \to 1$.  We have the following equalities by the geometry of the $\varphi_2$-polygon:

\begin{align*}
D_2 &= A_2 + E_2\\
B_2 &= E_2 + C_2 + D_2
\end{align*}

Thus,

\[
\frac{B_2}{A_2} = \frac{A_2 + C_2 + 2E_2}{A_2} = 1 + \frac{C_2}{A_2} + 2\frac{E_2}{A_2}
\]

and it remains to show that $\frac{E_2}{A_2}\to 0$.  

\begin{flalign*}
E_2 &= \frac1{\sqrt{2}}\int_0^b (z-a)^{-1/2} z^{-1/4} (b - z)^{-1/4}(1-z)^{-1/4} dz &&\\
&\approx (-a)^{-1/2}\int_0^bz^{-1/4} (b - z)^{-1/4}(1-z)^{-1/4} dz\\
&\approx (-a)^{-1/2}
\end{flalign*}

and $\frac{E_2}{A_2}\to 0$.

\end{proof}

\section{Embeddedness of the Families}\label{sec:embed}

In this section, we will prove the embeddedness of the families SS-4 and CS-4. Both cases are similar, we will therefore focus on CS-4. 
The main reason that the argument below works is that the degree of the Gauss map of our surfaces is 3 so that we can easily eliminate the possibility of critical points of the coordinate functions. We will prove:
\begin{enumerate}
\item The boundary of the CS-4 pentagon is a graph over a simple curve in the $yz$-plane, thus bounding a simply connected domain $\Omega$;
\item The  CS-4 pentagon is contained in an $x$-cylinder over $\Omega$;
\item The projection of the interior of the CS-4 pentagon has the unique path and homotopy lifting property.
\end{enumerate}

Then it follows that the CS-4 pentagon is a graph over $\Omega$: Otherwise, take a curve on the CS-4 pentagon that connects 
two points $x$-above a point $p\in \Omega$ and project it onto $\Omega$. Its image is closed in $\Omega$ and can be retracted onto $p$.
By the homotopy lifting property, the endpoints of the lifted curves need to be the same, contradicting that here are two points above $p$.

\smallskip

For (1), we note that the arcs $f(1)f(\infty)$ and $f(0)f(b)$ are diagonals of the box and hence trivially graphs over  segments in the plane $x=0$. The two arcs $f(b)f(1)$ and $f(\infty)f(a)$ lie in faces of the box parallel to the plane y=0. If either  of them is not graphical over a segment (parallel to the $z$-axis) in the plane $x=0$, then the Gauss map needs to become vertical at an interior point of that segment. This contradicts that $\deg G=3$, because all points with vertical normal are accounted for as vertices of the CS-4 pentagon and its reflective copies. It remains to discuss the arc $f(a)f(0)$ lying in a plane parallel to the plane $x=0$. 

Firstly, this curve is graphical over the line segment $f(a)f(0)$ (parallel to the $z$-axis), again by the same degree argument. We need to show that it stays away from the edges of the box face it lies in. The two edges parallel to the $y$-axis cannot be met because otherwise this produces points with vertical normal. 

For the edge through $f(a)f(0)$, first note  that the arc $f(a)f(0)$ stays near $f(a)$ and $f(0)$ inside the box face. If it leaves the box face across the segment $f(a)f(0)$, this would produce three points with normal in the $y$-direction, again contradicting (after replication) that $\deg G =3$. 

Finally we show that arc $f(a)f(0)$ stays on one side of the vertical box edge through $f(1)$. To see this, choose a point $f(x)$ on $f(a)f(0)$ and a point $f(y)$ on $f(b)f(\infty)$. Because the latter lies in the same plane parallel to $xz$-plane as $f(1)$, it suffices to show that the $y$-coordinate of $f(y)-f(x)$ never changes sign. This $y$-coordinate is equal to $\re \left(i (\Phi_1(y)-\Phi_1(x))+ i (\Phi_2(y)-\Phi_2(x)\right)$. From Figure \ref{fig:CS-4-schwarz} we see that $\Phi_1(y)-\Phi_1(x)$ and $\Phi_2(y)-\Phi_2(x)$ have positive imaginary parts regardless of the choice of $x$ or $y$, which implies the claim.

\smallskip
For (2), let's assume that the projection $\Pi$ of the  CS-4 pentagon onto the plane $x=0$ does not lie in $\Omega$. Take a boundary point of $\Pi$ that does not lie in $\overline \Omega$. By the implicit function theorem, its preimage in the CS-4 pentagon has a  normal in the $yz$-plane.  To see why this is impossible, we use the explicit form of the Gauss map
\[
G(z)^2 = \frac{z-a}{\sqrt{1-z}\sqrt{z} \sqrt{z-b}}
\]
and show that its square it is never negative real. Recall that $z\mapsto\sqrt{1-z}\sqrt{z}$ maps the upper half plane to the right half plane. The same is true for \[
z\mapsto  \frac{z-a}{ \sqrt{z-b}} = \sqrt{z-b} +\frac{b-a}{\sqrt{z-b}}
\]
because $b-a>0$. Thus $G(z)^2$ is the quotient of two complex numbers in the right half plane and so never negative real.

\smallskip

For (3), we again use that the  interior of the  CS-4 pentagon has no points with horizontal normal. The claim follows from the implicit function, applied in the compact region where the curve of homotopy resides.

For the SS-4 surface, the argument is similar, using the plane $y=0$. 

\section{A Characterization of $SS$-4 and $CS$-4 by Symmetries}\label{sec:sym}

We conclude this paper with a uniqueness statement:
 
 \begin{theorem}
 Let $X$ be a genus 4 triply periodic minimal surface such that 
\begin{enumerate}
\item the planes $x=0$ and $y=0$ are symmetry planes;
\item the surface contains the line $y=x$ in the plane $z=0$;
\item a translational fundamental domain of $X$ is bounded by the planes $x=\pm1$,  $y=\pm 1$ and $z=\pm h$; 
\item this domain is cut into eight congruent copies by the planes $x=0$, $y=0$ and the line $y=x$ in the plane $z=0$.
\end{enumerate} 
Then $X$ is either of type CS or SS.
\end{theorem}

\begin{proof}
%
%
%

The eight congruent copies are simply connected minimal polygons, because otherwise the genus of $X$  would be at least 8. We call the box $[0,1]\times[0,1]\times[0,h]$ just ``the box'' in what follows.

The polygon vertices can only occur at the intersection of symmetry lines.
 Four of them must occur at the four polygon vertices where the horizontal lines intersect the vertical symmetry planes. These are $45^\circ$ vertices.
 All other vertices are $90^\circ$ vertices. Let's assume there are $n$ of such, so our polygon is a $(4+n)$-gon. Then there are  $4+2n$ vertices on the 
genus 4 surface and $4(4+n)$ edges. By Euler's theorem, we have necessarily $n=1$. Hence our polygons are pentagons.

We next locate points with vertical normal. They must occur at the pentagon vertices by symmetry, and on $X$ we have 6 of them. As the degree of the Gauss map of a genus 4 triply periodic minimal surface is 3, we have found all of them.

 This limits the  possibilities for the shape of the pentagon. There must be a single edge of the pentagon that connects  one end point of the top horizontal segment to an end point of the bottom horizontal segment. In addition, this segment must stay in a single face of the box.

To locate the remaining two arcs, we distinguish the two cases. If the two horizontal segments are parallel, we can arrange them as  in case SS of Figure \ref{fig:fundamental}. We can also assume that the vertices there labeled $f(b)$ and $f(\infty)$ are connected by the single edge that stays without loss in the front side of the box as in that figure. The remaining two edges must connect the two remaining vertices (labeled $f(0)$ and $f(1)$) of the horizontal segments. The vertex they have in common ($f(a)$) must occur on the vertical box segment connecting $f(0)$ and $f(1)$. This leaves two choices for the two polygon edges, that are, however, symmetric by a reflection of the box at a horizontal plane. Thus we are precisely in the situation of type SS.

Similarly, if the two horizontal segments are orthogonal, we can arrange them as  in case CS of Figure \ref{fig:fundamental}. We can also assume that the vertices labeled $f(b)$ and $f(1)$ are connected by the single edge that stays without loss in the left  front side of the box as in that figure. The remaining vertex must occur on either of the vertical box segments through $f(0)$ or $f(\infty)$, as any other choice would require more vertices. Without loss, we can assume that $f(a)$ occurs above $f(0)$. This forces the segment $f(a)f(\infty)$ to lie in the right back face of the box, and consequently $f(0)f(a)$ in the left back face.

\end{proof}

We end with a question: Can one classify  all genus 4 triply periodic minimal surface such that the planes $x=0$ and $y=0$ are symmetry planes and the surface contains the line $y=x$ in the plane $z=0$?

\bibliography{minlit}
\bibliographystyle{alpha}

\end{document}